\documentclass[journal,twoside,web]{ieeecolor}
\usepackage{lcsys}
\usepackage{cite}
\usepackage{amsmath,amssymb,amsfonts}
\usepackage{algorithmic}
\usepackage{graphicx}
\usepackage{textcomp}

\usepackage{soul,framed} 
\usepackage{bm}
\usepackage{cite}
\usepackage[linktocpage=true]{hyperref}
\usepackage{multicol,lipsum}
\usepackage{blkarray}
\usepackage{amsmath,amsfonts,amssymb}
\usepackage{mathtools}
\usepackage{dsfont}
\usepackage[capitalise]{cleveref}
\usepackage{tikz}
\usepackage{pgfplots}
\usepackage{graphicx}
\usepackage{subcaption}
\usepackage{wrapfig}
\usepackage{color}
\usepackage[linktocpage=true]{hyperref}
\usepackage{multicol,lipsum}
\usepackage{amsmath,amsfonts,amssymb, mathalpha}
\usepackage{mathtools}
\usepackage[capitalise]{cleveref}
\usepackage{dsfont}
\usepackage{tikz}
\usepackage{pgfplots}
\usepackage{graphicx}
\usepackage{subcaption}
\usepackage{wrapfig}
\usepackage{algorithmic}
\usepackage{textcomp}

\newcommand{\Z}{\mathds{Z}}
\newcommand{\R}{\mathds{R}}

\newcommand{\A}{\bm{\mathcal{A}}}

\newcommand{\X}{\bm{\mathcal{X}}}
\newcommand{\Y}{\bm{\mathcal{Y}}}
\newcommand{\U}{\bm{\mathcal{U}}}
\newcommand{\D}{\bm{\mathcal{D}}}
\newcommand{\f}{\bm f}
\newcommand{\g}{\bm g}

\renewcommand{\t}{\bm\theta}
\renewcommand{\a}{\bm\alpha}
\renewcommand{\u}{\bm u}
\newcommand{\x}{\bm x}
\newcommand{\y}{\bm y}
\newcommand{\e}{\bm e}
\renewcommand{\b}{\bm b}
\newcommand{\h}{\bm h}
\renewcommand{\A}{\bm A}
\renewcommand{\r}{\bm r}

\newtheorem{theorem}{Theorem}
\newtheorem{lemma}{Lemma}

\newtheorem{remark}{Remark}

\newtheorem{definition}{Definition}
\newtheorem{proposition}{Proposition}
\newtheorem{example}{Example}

\newcommand{\mtj}{\mathtt{j}}
\newcommand{\mtx}{\mathtt{x}}
\newcommand{\mtu}{\mathtt{u}}
\newcommand{\mty}{\mathtt{y}}
\newcommand{\mtt}{\mathtt{\uptheta}}
\newcommand{\mtn}{\mathtt{n}}
\newcommand{\mtm}{\mathtt{m}}
\newcommand{\mti}{\mathtt{i}}

\usepackage{upgreek}

\def\BibTeX{{\rm B\kern-.05em{\sc i\kern-.025em b}\kern-.08em
    T\kern-.1667em\lower.7ex\hbox{E}\kern-.125emX}}

\begin{document}
\title{Extending Identifiability Results from Continuous to Discrete-Space Systems}

\author{Anuththara Sarathchandra, Azadeh Aghaeeyan, and Pouria Ramazi
\thanks{A. Sararthchandra, A. Aghaeeyan, and P. Ramazi are with the Department of Mathematics and Statistics, Brock University, Canada (e-mail: $\{$alekamalage, aaghaeeyan, pramazi$\}$@brocku.ca).}
 }
\maketitle
\thispagestyle{empty} 

\begin{abstract}
    Researchers develop models to explain the unknowns.
    These models typically involve parameters that capture tangible quantities, the estimation of which is desired.
    Parameter identifiability investigates the recoverability of the unknown parameters given the error-free outputs, inputs, and the developed equations of the model. 
    Different notions of and methods to test identifiability exist for dynamical systems defined in the continuous space.
    Yet little attention was paid to the identifiability of discrete space systems, where variables and parameters are defined in a discrete space.
    We develop the identifiability framework for discrete space systems and highlight that this is not an immediate extension of the continuous space framework. 
    Unlike the continuous case, local identifiability concepts are sensitive to how a ``neighborhood'' is defined.
    Moreover, results on algebraic identifiability that proved useful in the continuous space are less so in their discrete form as the notion of differentiability disappears. 
\end{abstract}


\section{Introduction}
\label{sec:introduction}
\IEEEPARstart{T}{he} possibility of uniquely determining the parameters of a dynamical system by observing the inputs and outputs defines the \emph{identifiability} of that system. 
This property was investigated in, for example, biology \cite{raue2014comparison,wu2008parameter}, economics \cite{basmann1966application}, epidemiology  \cite{aghaeeyan2023revealing}, and control theory \cite{karlsson2012efficient, ramazi2014variance}.
In continuous-space systems, identifiability is defined \emph{analytically} by \emph{output equality}, that is whether the equality of two output trajectories implies that of the parameters \cite{walter1997identification}.
In \cite{grewal1976identifiability}, the same idea was
formulated as \emph{output distinguishability}, that is whether two different parameter values result in different output trajectories.
A useful technique to investigate the identifiability is to use \emph{differential algebra}, resulting in the notion of \emph{algebraic identifiability}, that is whether the parameter values can be expressed uniquely in terms of the inputs, outputs, and their time derivatives \cite{eisenberg2013identifiability,ljung1994global,saccomani2003parameter}.



Despite the wide applications of continuous space systems, there are cases where the state and parameter spaces are restricted to a discrete space. 
There are, however, challenges to extend the existing identifiability concepts and results in the continuous space to the discrete space.
Systems defined in discrete spaces lose the useful properties of differentiability.
Thus, some available approaches to investigate identifiability based on, for example, the Jacobian matrix, fail in discrete space systems.
Furthermore, in the continuous case, local identifiability notions are not sensitive to how the ``neighborhoods'' are defined. 
It suffices to show that the system is identifiable in a small enough neighborhood.
This is often not the case with the discrete case, because the neighborhood of a point cannot be arbitrarily small. 
Thus, local identifiability notions depend on the topology of a neighborhood in the discrete space, which is not uniquely defined. 

We aim to extend identifiability notions and methods from continuous to discrete space systems.
Our contribution is four-fold:
First, we justify the need for a separate framework for the discrete space and explain the subtleties associated with the notion of a discrete neighborhood--\cref{sec_problemFormulation}.
Second, we show that analytical definitions of identifiability, which are based on the output equality approach, can be adjusted to the discrete-space systems with minor changes--\cref{section Identifiability using output equality}.
Third, we develop the discrete space algebraic identifiability definitions and show that they are not a ready extension of their continuous space counterpart, especially those that depend on differentiation, such as the Jacobian, and that they may not be as useful--\cref{section Algebraic identifiability}.
Finally, we show that algebraic and local structural identifiability imply each other (\cref{theorem_relation between ALg Id and structural Id}) and provide the results that guarantee global identifiability based on the input-output equation--\cref{sec_results}.

\section{Problem formulation} \label{sec_problemFormulation}
Consider the discrete time and discrete space systems
\vspace{-5pt}
\begin{equation}\label{equation_model}
    \left\{
    \begin{aligned}
        &\x(t+1)=\f(\x(t), \u(t), \t), \quad \x_0=\x(0) \\
        &\y(t)=\g(\x(t), \t),  
    \end{aligned}
    \right.
\end{equation}
where $t\in\Z_{\geq0}$ denotes time, 
$\x \in \X$,  $\u \in \U$, $\y \in \Y$, and
$\t \in \mathbf{\Theta}$ are respectively, the vector of states, inputs, outputs, and the parameters, 
$\X\subseteq\mathbb{R}^{\mtx}$, 
$\U\subseteq\mathbb{R}^{\mtu}$,
$\Y\subseteq\mathbb{R}^{\mty}$, and
$\mathbf{\Theta}\subseteq\mathbb{R}^{\mtt}$ are the state, input, output, and parameter space for some positive integers $\mtx, \mtu, \mty, \mtt$, and $\f:\X\to\X$ and $\g:\X\to\Y$.
We assume that spaces $\X,\U,\Y,\mathbf{\Theta}$ are \emph{discrete} in the metric space $(\mathbb{R}^{\mti},d)$, for $\mti=\mtx, \mtu, \mty, \mtt$, respectively, and $d$ is the Euclidean distance \cite{bryant1985metric}.
This means that for every element of the space, there exists some $\delta>0$ such that the distance to any other element of the space exceeds $\delta$.
Should $\delta$ be independent of the elements, the (metric) space is \emph{uniformly discrete}.
For example, the set $\{1,\tfrac{1}{2},\tfrac{1}{4},\frac{1}{8}\ldots\}$ is discrete but not uniformly discrete, and the set $\frac{1}{n}\mathbb{Z}$ is uniformly discrete for $n\in\Z_{>0}$. 

\begin{example}\label{ex_Coordinators}
    Consider a population of $n$ decision makers who over time $t=1,2,\ldots$ choose whether to adopt some innovation, e.g., solar panels or augmented reality. 
    Each individual is either a \emph{weak coordinator}, who decides to adopt if at least half of the population has already done so, or a \emph{strong coordinator}, who decides to adopt only if at least one-fourth of the population has already done so. 
    The population proportions of the strong and weak coordinators are unknown and parameterized by $\theta_1$ and $\theta_2$ and together by the vector $\t = [\theta_1,\theta_2]^\top$ which belongs to the discretized simplex
    $\mathbf{\Theta} = \{ [\theta'_1,\theta'_2]^\top: \theta'_1+\theta'_2 =1, \theta'_1,\theta'_2\in\{0,\tfrac{1}{n},\tfrac{2}{n},\ldots,1\} \}$.
    The population state is defined by $\x=[x_1,x_2]^\top$ where $x_1$ and $x_2$ are the numbers of weak and strong coordinators who have adopted the innovation, divided by $n$.
    The state space is $\X = \{ \x'\in\mathbf{\Theta}: x'_1\leq \theta_1, x'_2\leq\theta_2 \} $.
    Both spaces are uniformly discrete.
    At each time step, all individual simultaneously revise their decisions, resulting in the autonomous dynamics
    \vspace{-2pt}
    \begin{equation} \label{eq:Coordinators_Dynamics}
        \begin{cases}
             \x(t+1) = [\theta_1 h(\x(t);\frac{1}{4}), \theta_2 h(\x(t);\frac{1}{2})]^\top, \\
            y(t) =x_1(t)+x_2(t),
        \end{cases}
    \end{equation}
    where
    $ h(\x(t);\tau)$ is the step function which is equal to 1 if $x_1(t)+x_2(t) \geq \tau$ and 0 otherwise.
\end{example}

We are interested in the parameter \emph{identifiability} of system \eqref{equation_model}.
It is defined both globally and locally.
\emph{Global identifiability} requires the uniqueness of the parameter values in the whole parameter space $\mathbf{\Theta}$
whereas \emph{local identifiability} consider a neighborhood of a parameter value $\t \in \mathbf{\Theta}$ \cite{nomm2004identifiability}.
We first present some possible neighborhoods in the discrete space.

\subsection{Neighborhoods in the discrete space}
\begin{example}\label{example_Neighborhood-1}
    Consider the parameter space $\mathbf{\Theta} = \Z^2$.
    The neighborhoods of a point $\t_0$ in the space can be defined by using the notion of a \emph{discrete ball with radius $r$}, that is the set of all points that lie within the distance of $r$ from $\t_0$ denoted by
    $B(\t_0,r)=\{ \t \in \mathbf{\Theta} \mid \| \t_0-\t \| \leq r\}$.
    Different neighborhoods of the point $\t_0\in\Z^2$ are formed by changing the value of $r$ (\cref{fig:Type 2}).
\end{example}

Parameters often need to meet specific conditions for feasibility, leading to a structured parameter space.
\setcounter{example}{0}
\begin{example}[revisited]\label{example_Neighborhood_Simplex} 
    The neighborhood of any point $\t\in\mathbf{\Theta}$ is restricted to a subset of $\mathbf{\Theta}$.
    \cref{fig:Type 3} is an example for $n=4$.
    Then for the parameter value $\t = [\tfrac{1}{2}, \tfrac{1}{2}]^\top$, the neighborhood of radius $\tfrac{1}{\sqrt{8}}$ would be
    $B([\tfrac{1}{2},\tfrac{1}{2}]^\top,\tfrac{1}{\sqrt{8}})\cap\mathbf{\Theta}$, which consists of the three points $[\tfrac{3}{4}, \tfrac{1}{4}]^\top$, $[\tfrac{1}{2},\tfrac{1}{2}]^\top$, and $[\tfrac{1}{4},\tfrac{3}{4}]^\top$.
    Any wider neighborhood will match the whole quantified simplex.
\end{example} 

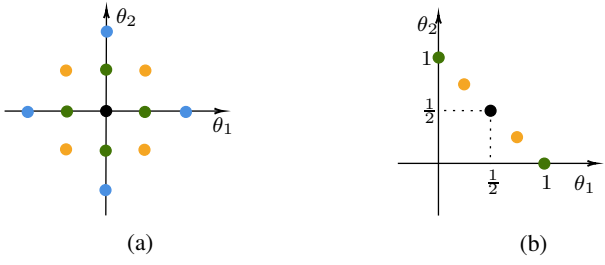
\begin{figure}[!tbp]
  \begin{subfigure}[b]{0.2\textwidth}
    \tikzset{every picture/.style={line width=0.5pt}} 
    \begin{tikzpicture}[x=0.5pt,y=0.5pt,yscale=-1,xscale=1]

\draw    (2.67,78) -- (166,78) ;
\draw [shift={(168,78)}, rotate = 180] [color={rgb, 255:red, 0; green, 0; blue, 0 }  ][line width=0.75]    (6.56,-1.97) .. controls (4.17,-0.84) and (1.99,-0.18) .. (0,0) .. controls (1.99,0.18) and (4.17,0.84) .. (6.56,1.97)   ;
\draw    (79.25,157) -- (79.74,3) ;
\draw [shift={(79.75,1)}, rotate = 90.18] [color={rgb, 255:red, 0; green, 0; blue, 0 }  ][line width=0.75]    (6.56,-1.97) .. controls (4.17,-0.84) and (1.99,-0.18) .. (0,0) .. controls (1.99,0.18) and (4.17,0.84) .. (6.56,1.97)   ;
\draw  [fill={rgb, 255:red, 0; green, 0; blue, 0 }  ,fill opacity=1 ] (75.41,77.85) .. controls (75.49,75.59) and (77.38,73.82) .. (79.65,73.89) .. controls (81.91,73.97) and (83.68,75.86) .. (83.61,78.13) .. controls (83.53,80.39) and (81.63,82.16) .. (79.37,82.09) .. controls (77.11,82.01) and (75.34,80.11) .. (75.41,77.85) -- cycle ;
\draw  [color={rgb, 255:red, 65; green, 117; blue, 5 }  ,draw opacity=1 ][fill={rgb, 255:red, 65; green, 117; blue, 5 }  ,fill opacity=1 ] (104.75,78.52) .. controls (104.82,76.26) and (106.72,74.48) .. (108.98,74.56) .. controls (111.24,74.64) and (113.02,76.53) .. (112.94,78.79) .. controls (112.86,81.06) and (110.97,82.83) .. (108.7,82.75) .. controls (106.44,82.68) and (104.67,80.78) .. (104.75,78.52) -- cycle ;
\draw  [color={rgb, 255:red, 65; green, 117; blue, 5 }  ,draw opacity=1 ][fill={rgb, 255:red, 65; green, 117; blue, 5 }  ,fill opacity=1 ] (75.41,46.52) .. controls (75.49,44.26) and (77.38,42.48) .. (79.65,42.56) .. controls (81.91,42.64) and (83.68,44.53) .. (83.61,46.79) .. controls (83.53,49.06) and (81.63,50.83) .. (79.37,50.75) .. controls (77.11,50.68) and (75.34,48.78) .. (75.41,46.52) -- cycle ;
\draw  [color={rgb, 255:red, 65; green, 117; blue, 5 }  ,draw opacity=1 ][fill={rgb, 255:red, 65; green, 117; blue, 5 }  ,fill opacity=1 ] (45.75,78.18) .. controls (45.82,75.92) and (47.72,74.15) .. (49.98,74.23) .. controls (52.24,74.3) and (54.02,76.2) .. (53.94,78.46) .. controls (53.86,80.72) and (51.97,82.5) .. (49.7,82.42) .. controls (47.44,82.34) and (45.67,80.45) .. (45.75,78.18) -- cycle ;
\draw  [color={rgb, 255:red, 65; green, 117; blue, 5 }  ,draw opacity=1 ][fill={rgb, 255:red, 65; green, 117; blue, 5 }  ,fill opacity=1 ] (75.08,107.85) .. controls (75.16,105.59) and (77.05,103.82) .. (79.31,103.89) .. controls (81.58,103.97) and (83.35,105.86) .. (83.27,108.13) .. controls (83.2,110.39) and (81.3,112.16) .. (79.04,112.09) .. controls (76.78,112.01) and (75,110.11) .. (75.08,107.85) -- cycle ;
\draw  [color={rgb, 255:red, 74; green, 144; blue, 226 }  ,draw opacity=1 ][fill={rgb, 255:red, 74; green, 144; blue, 226 }  ,fill opacity=1 ] (135.75,78.52) .. controls (135.82,76.26) and (137.72,74.48) .. (139.98,74.56) .. controls (142.24,74.64) and (144.02,76.53) .. (143.94,78.79) .. controls (143.86,81.06) and (141.97,82.83) .. (139.7,82.75) .. controls (137.44,82.68) and (135.67,80.78) .. (135.75,78.52) -- cycle ;
\draw  [color={rgb, 255:red, 74; green, 144; blue, 226 }  ,draw opacity=1 ][fill={rgb, 255:red, 74; green, 144; blue, 226 }  ,fill opacity=1 ] (16.08,78.52) .. controls (16.16,76.26) and (18.05,74.48) .. (20.31,74.56) .. controls (22.58,74.64) and (24.35,76.53) .. (24.27,78.79) .. controls (24.2,81.06) and (22.3,82.83) .. (20.04,82.75) .. controls (17.78,82.68) and (16,80.78) .. (16.08,78.52) -- cycle ;
\draw  [color={rgb, 255:red, 245; green, 166; blue, 35 }  ,draw opacity=1 ][fill={rgb, 255:red, 245; green, 166; blue, 35 }  ,fill opacity=1 ] (105.08,47.18) .. controls (105.16,44.92) and (107.05,43.15) .. (109.31,43.23) .. controls (111.58,43.3) and (113.35,45.2) .. (113.27,47.46) .. controls (113.2,49.72) and (111.3,51.5) .. (109.04,51.42) .. controls (106.78,51.34) and (105,49.45) .. (105.08,47.18) -- cycle ;
\draw  [color={rgb, 255:red, 245; green, 166; blue, 35 }  ,draw opacity=1 ][fill={rgb, 255:red, 245; green, 166; blue, 35 }  ,fill opacity=1 ] (45.08,47.18) .. controls (45.16,44.92) and (47.05,43.15) .. (49.31,43.23) .. controls (51.58,43.3) and (53.35,45.2) .. (53.27,47.46) .. controls (53.2,49.72) and (51.3,51.5) .. (49.04,51.42) .. controls (46.78,51.34) and (45,49.45) .. (45.08,47.18) -- cycle ;
\draw  [color={rgb, 255:red, 245; green, 166; blue, 35 }  ,draw opacity=1 ][fill={rgb, 255:red, 245; green, 166; blue, 35 }  ,fill opacity=1 ] (104.41,107.18) .. controls (104.49,104.92) and (106.38,103.15) .. (108.65,103.23) .. controls (110.91,103.3) and (112.68,105.2) .. (112.61,107.46) .. controls (112.53,109.72) and (110.63,111.5) .. (108.37,111.42) .. controls (106.11,111.34) and (104.34,109.45) .. (104.41,107.18) -- cycle ;
\draw  [color={rgb, 255:red, 245; green, 166; blue, 35 }  ,draw opacity=1 ][fill={rgb, 255:red, 245; green, 166; blue, 35 }  ,fill opacity=1 ] (45.08,106.85) .. controls (45.16,104.59) and (47.05,102.82) .. (49.31,102.89) .. controls (51.58,102.97) and (53.35,104.86) .. (53.27,107.13) .. controls (53.2,109.39) and (51.3,111.16) .. (49.04,111.09) .. controls (46.78,111.01) and (45,109.11) .. (45.08,106.85) -- cycle ;
\draw  [color={rgb, 255:red, 74; green, 144; blue, 226 }  ,draw opacity=1 ][fill={rgb, 255:red, 74; green, 144; blue, 226 }  ,fill opacity=1 ] (74.91,137.68) .. controls (74.99,135.42) and (76.88,133.65) .. (79.15,133.73) .. controls (81.41,133.8) and (83.18,135.7) .. (83.11,137.96) .. controls (83.03,140.22) and (81.13,142) .. (78.87,141.92) .. controls (76.61,141.84) and (74.84,139.95) .. (74.91,137.68) -- cycle ;
\draw  [color={rgb, 255:red, 74; green, 144; blue, 226 }  ,draw opacity=1 ][fill={rgb, 255:red, 74; green, 144; blue, 226 }  ,fill opacity=1 ] (75.91,17.68) .. controls (75.99,15.42) and (77.88,13.65) .. (80.15,13.73) .. controls (82.41,13.8) and (84.18,15.7) .. (84.11,17.96) .. controls (84.03,20.22) and (82.13,22) .. (79.87,21.92) .. controls (77.61,21.84) and (75.84,19.95) .. (75.91,17.68) -- cycle ;

\draw (158,80.4) node [anchor=north west][inner sep=0.75pt]  [font=\footnotesize]  {$\theta_1$};
\draw (85,-1.6) node [anchor=north west][inner sep=0.75pt]  [font=\footnotesize]  {$\theta_2$};

\end{tikzpicture}
    \caption{}
    \label{fig:Type 2}
  \end{subfigure}
  \hfill
  \begin{subfigure}[b]{0.2\textwidth}
    \tikzset{every picture/.style={line width=0.5pt}} 
    \begin{tikzpicture}[x=0.5pt,y=0.5pt,yscale=-1,xscale=1]
\draw    (0,120.5) -- (150,120.5) ;
\draw [shift={(150,120.5)}, rotate = 180] [color={rgb, 255:red, 0; green, 0; blue, 0 }  ][line width=0.75]    (6.56,-1.97) .. controls (4.17,-0.84) and (1.99,-0.18) .. (0,0) .. controls (1.99,0.18) and (4.17,0.84) .. (6.56,1.97)   ;
\draw    (30.74,160) -- (30.74,10) ;
\draw [shift={(30.75,8)}, rotate = 90.18] [color={rgb, 255:red, 0; green, 0; blue, 0 }  ][line width=0.75]    (6.56,-1.97) .. controls (4.17,-0.84) and (1.99,-0.18) .. (0,0) .. controls (1.99,0.18) and (4.17,0.84) .. (6.56,1.97)   ;
\draw  [color={rgb, 255:red, 65; green, 117; blue, 5 }  ,draw opacity=1 ][fill={rgb, 255:red, 65; green, 117; blue, 5 }  ,fill opacity=1 ] (106.75,120.52) .. controls (106.82,118.26) and (108.72,116.48) .. (110.98,116.56) .. controls (113.24,116.64) and (115.02,118.53) .. (114.94,120.79) .. controls (114.86,123.06) and (112.97,124.83) .. (110.7,124.75) .. controls (108.44,124.68) and (106.67,122.78) .. (106.75,120.52) -- cycle ;
\draw  [color={rgb, 255:red, 65; green, 117; blue, 5 }  ,draw opacity=1 ][fill={rgb, 255:red, 65; green, 117; blue, 5 }  ,fill opacity=1 ] (26.81,40.22) .. controls (26.88,37.96) and (28.78,36.19) .. (31.04,36.26) .. controls (33.3,36.34) and (35.08,38.24) .. (35,40.5) .. controls (34.92,42.76) and (33.03,44.54) .. (30.76,44.46) .. controls (28.5,44.38) and (26.73,42.49) .. (26.81,40.22) -- cycle ;
\draw  [color={rgb, 255:red, 245; green, 166; blue, 35 }  ,draw opacity=1 ][fill={rgb, 255:red, 245; green, 166; blue, 35 }  ,fill opacity=1 ] (46.08,60.52) .. controls (46.16,58.26) and (48.05,56.48) .. (50.31,56.56) .. controls (52.58,56.64) and (54.35,58.53) .. (54.27,60.79) .. controls (54.2,63.06) and (52.3,64.83) .. (50.04,64.75) .. controls (47.78,64.68) and (46,62.78) .. (46.08,60.52) -- cycle ;
\draw  [color={rgb, 255:red, 0; green, 0; blue, 0 }  ,draw opacity=1 ][fill={rgb, 255:red, 0; green, 0; blue, 0 }  ,fill opacity=1 ] (66.08,80.52) .. controls (66.16,78.26) and (68.05,76.48) .. (70.31,76.56) .. controls (72.58,76.64) and (74.35,78.53) .. (74.27,80.79) .. controls (74.2,83.06) and (72.3,84.83) .. (70.04,84.75) .. controls (67.78,84.68) and (66,82.78) .. (66.08,80.52) -- cycle ;
\draw  [color={rgb, 255:red, 245; green, 166; blue, 35 }  ,draw opacity=1 ][fill={rgb, 255:red, 245; green, 166; blue, 35 }  ,fill opacity=1 ] (86.08,100.52) .. controls (86.16,98.26) and (88.05,96.48) .. (90.31,96.56) .. controls (92.58,96.64) and (94.35,98.53) .. (94.27,100.79) .. controls (94.2,103.06) and (92.3,104.83) .. (90.04,104.75) .. controls (87.78,104.68) and (86,102.78) .. (86.08,100.52) -- cycle ;
\draw  [dash pattern={on 0.84pt off 2.51pt}]  (70.18,80.66) -- (70,121) ;
\draw  [dash pattern={on 0.84pt off 2.51pt}]  (70.18,80.66) -- (29.95,80.45) ;

\draw (106.33,127.74) node [anchor=north west][inner sep=0.75pt]  [font=\footnotesize]  {$1$};
\draw (15.67,33.74) node [anchor=north west][inner sep=0.75pt]  [font=\footnotesize]  {$1$};
\draw (64,123.4) node [anchor=north west][inner sep=0.75pt]  [font=\scriptsize]  {$\frac{1}{2}$};
\draw (16,69.4) node [anchor=north west][inner sep=0.75pt]  [font=\scriptsize]  {$\frac{1}{2}$};
\draw (12.75,7.4) node [anchor=north west][inner sep=0.75pt]  [font=\footnotesize]  {$\theta _{2}$};
\draw (131.5,127.4) node [anchor=north west][inner sep=0.75pt]  [font=\footnotesize]  {$\theta _{1}$};
\end{tikzpicture}
    \caption{}
    \label{fig:Type 3}
  \end{subfigure}
  \caption{\small
  \textbf{(a)} The neighborhood of $\t_0 = [0,0]^\top$ in $\Z^2$.
  \textbf{(b)} The one-dimensional simplex intersected with $\tfrac{1}{4}\Z^2$. 
    }
\end{figure}

\subsection{The need for a separate framework}
Instead of crafting a distinct framework to explore identifiability in discrete-space systems, one might propose substituting discrete-spaces with continuous ones and leveraging the existing continuous-space framework. 
However, the subsequent example underscores the limitations of this approach.
For simplicity, by identifiability here, we refer to the possibility of obtaining a unique value for the parameter of interest.
Rigorous definitions are provided in \cref{section Identifiability using output equality}.
\begin{example}\label{example_CounterExample_1_ForAlgebraicID}
    Consider the digital electrical circuit with n-bit precision in Fig \ref{fig_algebraic_example}, where D flip-flops are memory components that transfer data from $D$ to $Q$ at the rising edges of the clock signal ($\text{CLK}$). 
   Addition and multiplication are performed using blocks labeled by $+$ and $\times$, respectively.
   The square blocks represent n-bit latches used for storing constant parameter values.
    The system can be described by 
    \vspace{-2pt}
    \begin{equation*}
        \begin{cases}
            \x(t+1) = [x_1(t)^2 + \theta_2 x_2(t) + u(t), \theta_1 x_1(t)]^\top, \\
            y(t)=x_1(t), 
        \end{cases}
    \end{equation*}
    where $\x \in \Z^2$, $\t \in \Z^2_{\geq 0}$ with $\x(0)=[2,1]$, $u(0) = 3$, $u(1) = 0$ and we observe the outputs $y(0)=2$, $y(1)=8$, $y(2)=66$.
    The system equations yield
    \vspace{-3pt}
    \begin{equation} \label{example_CounterExample_1_ForAlgebraicID_eq1}
        y(t+2) = y(t+1)^2 + \theta_1\theta_2 y(t) + u(t+1),
    \end{equation}
    which gives $\theta_1\theta_2=(y(2)-y(1)^2 - u(1))/y(0)= 1$.
   The parameter values can only be non-negative integers as the n-bit latches are not storing floating points.
    Hence, the parameters can be recovered uniquely as $\theta_1 = \theta_2 = 1$.
    However, if the parameter space was not restricted, infinite values for $\theta_1$ and $\theta_2$ would be possible.
\end{example}

\begin{figure}
    \centering  

\tikzset{every picture/.style={line width=0.75pt}} 

\tikzset{every picture/.style={line width=0.75pt}} 

\begin{tikzpicture}[x=0.75pt,y=0.75pt,yscale=-1,xscale=.9]

\draw   (184,23) -- (199,23) .. controls (207.28,23) and (214,28.96) .. (214,36.3) .. controls (214,43.64) and (207.28,49.6) .. (199,49.6) -- (184,49.6) -- (184,23) -- cycle (174,27.43) -- (184,27.43) (174,45.17) -- (184,45.17) (214,36.3) -- (224,36.3) ;
\draw   (127,15.6) -- (127.21,78.61) -- (94.18,78.72) -- (93.98,15.71) -- cycle ;
\draw    (148,27.6) -- (174,27.43) ;
\draw    (128,36.6) -- (148.08,36.67) ;
\draw    (148,45.6) -- (174,45.17) ;
\draw    (148,27.6) -- (148,45.6) ;
\draw   (117.79,132.64) -- (118,196.6) -- (82.21,196.72) -- (82,132.76) -- cycle ;
\draw    (117,146.6) -- (155.8,146.87) ;
\draw   (125,156.6) -- (145.4,156.6) -- (145.4,177) -- (125,177) -- cycle ;
\draw    (145.4,164.6) -- (158.4,164.6) ;
\draw   (232.6,77) -- (243.6,77) .. controls (251.27,77.25) and (258.13,82.63) .. (261.2,90.8) .. controls (258.13,98.97) and (251.27,104.35) .. (243.6,104.6) -- (232.6,104.6) .. controls (237.32,96.06) and (237.32,85.54) .. (232.6,77) -- cycle (226,81.6) -- (234.8,81.6) (226,100) -- (234.8,100) (261.2,90.8) -- (270,90.8) ;
\draw    (224,36.3) -- (224.47,81.77) ;
\draw    (214,37) -- (224,36.6) ;
\draw    (221,141) -- (221,99.6) ;
\draw    (205.75,141) -- (221,141) ;
\draw    (314.5,11.5) -- (314,123.8) ;
\draw    (84.08,11.17) -- (314.5,11.5) ;
\draw    (84.08,11.17) -- (84,37.6) ;
\draw    (84,37.6) -- (94,37.6) ;
\draw    (148,45.6) -- (148.16,92.28) ;
\draw   (177.03,119.96) -- (156.63,120.01) -- (156.57,99.61) -- (176.97,99.56) -- cycle ;
\draw    (156,110.01) -- (143,110.04) ;
\draw    (205.75,141) -- (205.8,155.73) ;
\draw    (71,100.6) -- (98.11,100.82) ;
\draw    (71,100.6) -- (71,149.6) ;
\draw    (71,149.6) -- (81.67,149.75) ;
\draw   (165.8,142.43) -- (180.8,142.43) .. controls (189.08,142.43) and (195.8,148.39) .. (195.8,155.73) .. controls (195.8,163.07) and (189.08,169.03) .. (180.8,169.03) -- (165.8,169.03) -- (165.8,142.43) -- cycle (155.8,146.87) -- (165.8,146.87) (155.8,164.6) -- (165.8,164.6) (195.8,155.73) -- (205.8,155.73) ;
\draw   (138.02,114.38) -- (123.02,114.28) .. controls (114.74,114.23) and (108.06,108.23) .. (108.11,100.89) .. controls (108.15,93.55) and (114.91,87.63) .. (123.19,87.68) -- (138.19,87.78) -- (138.02,114.38) -- cycle (148.05,110.01) -- (138.05,109.95) (148.16,92.28) -- (138.16,92.21) (108.11,100.89) -- (98.11,100.82) ;
\draw   (94.18,78.72) -- (103.74,73.63) -- (94.09,68.72) -- (94.18,78.72) -- cycle (89.34,73.76) -- (94.14,73.72) ;
\draw   (82.21,196.72) -- (91.76,191.63) -- (82.12,186.72) -- (82.21,196.72) -- cycle (77.36,191.76) -- (82.16,191.72) ;
\draw    (35.34,73.16) -- (89.34,73.76) ;
\draw    (36,190.6) -- (77.36,191.76) ;
\draw    (35.34,73.16) -- (36,190.6) ;
\draw    (15.42,130.75) -- (35.67,130.88) ;
\draw    (220.75,100.1) -- (226,100) ;
\draw    (224.47,81.77) -- (226,81.6) ;
\draw   (276.6,110) -- (287.6,110) .. controls (295.27,110.25) and (302.13,115.63) .. (305.2,123.8) .. controls (302.13,131.97) and (295.27,137.35) .. (287.6,137.6) -- (276.6,137.6) .. controls (281.32,129.06) and (281.32,118.54) .. (276.6,110) -- cycle (270,114.6) -- (278.8,114.6) (270,133) -- (278.8,133) (305.2,123.8) -- (314,123.8) ;
\draw    (270,90.8) -- (270,114.6) ;
\draw    (270,133) -- (255.17,133) ;
\draw [color={rgb, 255:red, 0; green, 0; blue, 0 }  ,draw opacity=1 ]   (148.08,68.94) -- (156.91,68.94) ;

\draw (158.57,103.01) node [anchor=north west][inner sep=0.75pt]  [font=\scriptsize]  {$\theta _{1}$};
\draw (127,160) node [anchor=north west][inner sep=0.75pt]  [font=\scriptsize]  {$\theta _{2}$};
\draw (7,24) node [anchor=north west][inner sep=0.75pt]  [font=\scriptsize] [align=left] {};
\draw (103.73,78.62) node [anchor=north west][inner sep=0.75pt]  [font=\scriptsize,rotate=-269.81] [align=left] {{\scriptsize n-bit D flip-flop}};
\draw (131,15.4) node [anchor=north west][inner sep=0.75pt]  [font=\scriptsize]  {$x_{1}( k)$};
\draw (124,131.4) node [anchor=north west][inner sep=0.75pt]  [font=\scriptsize]  {$x_{2}( k)$};
\draw (12,119.5) node [anchor=north west][inner sep=0.75pt]  [font=\scriptsize] [align=left] {{\scriptsize CLK}};
\draw (92.76,196.62) node [anchor=north west][inner sep=0.75pt]  [font=\scriptsize,rotate=-269.81] [align=left] {{\scriptsize n-bit D flip-flop}};
\draw (189,30.83) node [anchor=north west][inner sep=0.75pt]  [font=\scriptsize]  {$\times $};
\draw (172,150.83) node [anchor=north west][inner sep=0.75pt]  [font=\scriptsize]  {$\times $};
\draw (116,95.83) node [anchor=north west][inner sep=0.75pt]  [font=\scriptsize]  {$\times $};
\draw (236.8,85) node [anchor=north west][inner sep=0.75pt]  [font=\scriptsize]  {$+$};
\draw (95,33) node [anchor=north west][inner sep=0.75pt]  [font=\scriptsize] [align=left] {D};
\draw (116.84,32.17) node [anchor=north west][inner sep=0.75pt]  [font=\scriptsize,rotate=-358.09] [align=left] {Q};
\draw (82,145) node [anchor=north west][inner sep=0.75pt]  [font=\scriptsize] [align=left] {D};
\draw (107.84,142.17) node [anchor=north west][inner sep=0.75pt]  [font=\scriptsize,rotate=-358.09] [align=left] {Q};
\draw (286,119.4) node [anchor=north west][inner sep=0.75pt]  [font=\scriptsize]  {$+$};
\draw (231.83,127) node [anchor=north west][inner sep=0.75pt]  [font=\scriptsize] [align=left] {{\scriptsize Input}};
\draw (158.83,62.4) node [anchor=north west][inner sep=0.75pt]  [font=\scriptsize]  {$y(k)$};
\end{tikzpicture}

\caption{\small A digital electrical circuit. 
The D flip-flop device makes the digital circuit sequential, i.e., the value at the Q port of a D flip-flop equals that of the D port at the previous time step.
The clock signal $\text{CLK}$ synchronizes the two flip-flops.
The parameters are stored in n-bit latches.
In practice, all operations are performed at the bit level; here, all devices operate using n-bit precision.}
\label{fig_algebraic_example}
\end{figure}
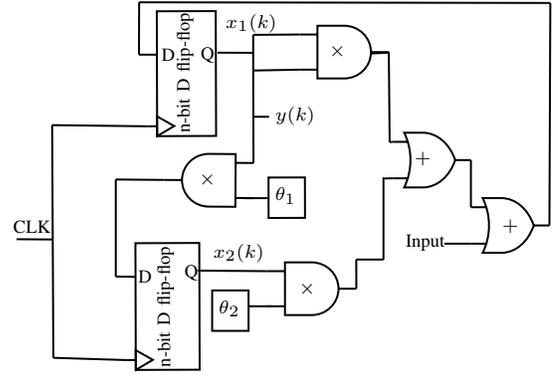
Now we proceed to the analytical definitions of identifiability in the discrete-time and space, which are often identical to those of discrete-time and continuous-space case.

\section{Identifiability: analytical definitions} \label{section Identifiability using output equality}
The main definition of identifiability is based on output equality.
The idea behind this definition is that, for identifiable parameters, same output trajectories cannot be generated from identical systems that differ only in their parameter values. 
In other words, a parameter value is identifiable if it is  \emph{distinguishable} from other values, that is, if they result in different output (solution) trajectories \cite{grewal1976identifiability}. 
In continuous-time systems, the notion of solution trajectory is defined as the set of pairs $(t,\y(t))$ for $t\in\mathbb{R}$. 
In discrete-time systems, one can use the output sequence denoted by $\langle \y(t)\rangle_{t=0}^T$ or simply $\langle \y(t)\rangle_{0}^T$ for some $T>0$. 
To emphasize the role of the initial condition, input, and parameter in the output, we use the notation 
$\langle\y(t,\x_0, \langle\u(t')\rangle_{0}^{t-1}, \t)\rangle_{0}^T$ or simply
$\langle\y(t,\x_0, \langle\u(t)\rangle_{0}^{T-1}, \t)\rangle_{0}^T$.

\subsection{Local identifiability}
The notion of locally strong identifiability was first introduced in \cite{tunali1987new} for continuous-time and space systems. 
The following is the equivalent discrete-time and space with minor changes, including the initial time being set to $0$ rather than a general $t_0$ value.
For integer $T>0$, define the set power $\U^T=\U\times\ldots\times\U$.
\begin{definition}
\label{definition_LocalId}
    System \eqref{equation_model} is \textbf{locally strongly $\x_0$-identifiable at $\t\in\mathbf{\Theta}$} 
    if there exists a neighborhood $\mathbf{\Theta}'\subset\mathbf{\Theta}$ of $\t$, such that for any $\t,\t'\in \mathbf{\Theta}'$ 
    there is an input sequence $\u^*=\langle\u(t)\rangle_0^{T-1}\in\U^{T}$, such that for every time $T>0$, 
    \begin{equation}\label{Eq_distinguishability}
        \scalebox{.98}{$
        \t \neq \t'
        \Rightarrow 
        \langle \y(t,\x_0,\u^*,\t)\rangle_0^T 
            \neq \langle \y(t,\x_0,\u^*,\t')\rangle_0^T.
            $}
    \end{equation}
\end{definition}
\begin{remark}\label{re_theDomainofUinLocallyStrongIdentifiability}
     In \eqref{equation_model}, the output is a function of state $\x(t)$, not the input.
    Moreover, $\x(t)$ depends on $\u(t-1)$ not $\u(t)$.
    Thus, the input sequence $\langle\u(t)\rangle_0^{T-1}$ together with the initial condition $\x_0$ and parameter $\t$ uniquely determine the output sequence $\langle \y(t,\x_0,\u^*,\t)\rangle_0^T$.
\end{remark}

Since \eqref{Eq_distinguishability} must hold for any $T>0$ in \cref{definition_LocalId}, it must also hold for $T=1$, i.e., the smallest possible value of $T \in \Z_{\geq 0}$. 
Moreover, if it holds for $T=1$, then it also holds for any $T>0$. 
Thus, in view of \cref{re_theDomainofUinLocallyStrongIdentifiability},
\cref{definition_LocalId} can be simplified to the following, implying that the output distinguishability should only be checked for $t=0$ and 1.
\begin{definition}
\label{definition_LocalId_2}
    System \eqref{equation_model} is \textbf{locally strongly $\x_0$-identifiable at $\t\in\mathbf{\Theta}$} 
    if there exists a neighborhood $\mathbf{\Theta}'\subset\mathbf{\Theta}$ of $\t$, such that for any $\t,\t'\in \mathbf{\Theta}'$ 
     there is an input $\u(0)\in\U$, such that \eqref{Eq_distinguishability} holds for $T=1$.   
\end{definition}
\begin{remark}\label{re_theValueofTinLocallyStrongIdentifiability}
    The simplification in \cref{definition_LocalId_2} is merely due to the discrete-time nature of the dynamics. 
    This definition is restrictive since the dimension of the parameter space can be at most two times the dimension of the output, i.e., $\mtt\leq2\mty$ in order for the system to be identifiable. 
    This motivates the less restrictive \cref{de_identifiabilityThrough}.
\end{remark}

\begin{definition}[\cite{nomm2016further,anstett2020priori}] \label{de_identifiabilityThrough}
     System \eqref{equation_model} is \textbf{locally strongly $\x_0$-identifiable at $\t\in\mathbf{\Theta}$ through the input sequence $\u^*=\langle\u(t)\rangle_{t=0}^{T-1}\in\U^T$} for some $T>0$,
    if there exists a neighborhood $\mathbf{\Theta}'\subset\mathbf{\Theta}$ of $\t$, such that for any $\t,\t'\in \mathbf{\Theta}'$, \eqref{Eq_distinguishability} holds.
\end{definition}

The following example shows how a system can be locally identifiable only at some initial conditions.
It also highlights the role of the neighborhood.
\begin{example} \label{example_stronglyLocallyIdentifiable}
    Consider the discrete-time and discrete-space system defined by $x(t+1)= \theta^2 x(t)$ and $y(t)=  x(t)$,
    where $x,y,\theta \in \Z$.
    For $x_0=0$, $\theta \neq \theta'$ does not imply $y(t,x,\theta)\neq y(t,x,\theta')$. 
    Hence, for all $\theta \in {\Theta}$, the system is not locally strongly $0$-identifiable at $\theta$.
    Nevertheless, for all other $x_0\neq0$, the systems is locally strongly $x_0$-identifiable at any $\theta\neq0$ in $B(\theta,1)$.
    The only case where two different $\theta$ and $\theta'$ can result in the same output trajectory is when $\theta=-\theta'$; however, the two then do not lie in the same ball $B(\theta,1)$ for $\theta\neq0$.
    The same is not concluded for the neighborhood $B(\theta,2)$ as then for $\theta=1$, the ball includes the non-distinguishable parameter values $1$ and $-1$.
\end{example} 

\subsection{Local structural identifiability}
The notion of locally strong $\x_0$-identifiability is defined for a given $\x_0$.
The notion of \emph{structural identifiability} requires the system to be strongly locally $\x_0$-identifiable for all initial conditions $\x_0$ in the state space except for those of measure zero \cite{eisenberg2013input,saccomani2003parameter, anstett2020priori}. 
It additionally requires the existence of some input space $\U^T$ such that the system is strongly locally $\x_0$-identifiable for every input sequence in $\U^T$.
A weak notion of structural identifiability was presented in \cite{nomm2004identifiability,nomm2016further} where instead of the whole state space, it requires identifiability for initial conditions in a dense subset of the state space only. 
\begin{definition}
\label{definition_StructuralID}
    System \eqref{equation_model} is \textbf{locally structurally identifiable} 
    if there exist a $T>0$ and subsets 
    $\mathbf{\Theta}'\subseteq\mathbf{\Theta}$, 
    $\X'\subseteq\X$, and
    $\U'\subseteq\U^T$,
    such that the system is locally strongly $\x_0$-identifiable at $\t$ through the input sequence $\langle\u(t)\rangle_0^{T-1}$ 
    for almost all
    $\t\in\mathbf{\Theta}'$, 
    $\x_0 \in \X'$, and 
    $\langle\u(t)\rangle_0^{T-1}\in\U'$.    
\end{definition} 



\subsection{Global identifiability}
\begin{definition}[\cite{saccomani2003parameter}]
    System \eqref{equation_model} is \textbf{globally identifiable} at $\t$
    if there exist a $T>0$ and an input sequence $\u^*=\langle\u(t)\rangle_0^{T-1}$ such that \eqref{Eq_distinguishability} holds for all $\t' \in \mathbf{\Theta}$ and all $\x_0\in\X$.
\end{definition}

\setcounter{example}{0}
\begin{example} [revisited]
    \label{ex_local_ID_of_Coordinting_Example}
    Let $n=4$.
    For $\x_0=[\frac{1}{4},0]^\top$ or $[0,\frac{1}{4}]^\top$,
    the output trajectory equals $\langle \frac{1}{4}, \theta_1, \theta_1+\theta_2, \theta_1+\theta_2, \ldots \rangle$ 
    if $\theta_1\geq \frac{1}{2}$ and $\langle \frac{1}{4}, \theta_1, \theta_1, \ldots \rangle$ otherwise, 
    which are unique for each $\t \in \mathbf{\Theta}$ implying that the system is locally strongly $\x_0-$identifiable at $\t$ for all $\t \in \mathbf{\Theta}$.
    It follows that the system is locally structurally identifiable (by setting $\mathbf{\Theta}' = \mathbf{\Theta}$ and $\X'=\{[\frac{1}{4},0],[0,\frac{1}{4}]\}$).
    This is not the case for other initial conditions as the output trajectory becomes the same for any $\t \in \mathbf{\Theta}$, and equals either 
    $\langle y(0),0,0,\ldots \rangle$ or $\langle y(0), \theta_1+\theta_2, \theta_1+\theta_2, \ldots \rangle$.
    Consequently, the system is not globally identifiable.
\end{example}


\section{Identifiability: algebraic definitions} \label{section Algebraic identifiability}
In analytic definitions, we infer about the equality of the parameter values using the equality of the corresponding two output trajectories.
The notion of \emph{algebraic identifiability} instead is based on constructing the so-called \emph{input-output equation}, also referred to as the \emph{algebraic equation}, in the form of $\bm{\Phi} = 0$, where $\bm{\Phi}$ 
consists of the parameters and time iterations of inputs and outputs. 
Then, the possibility of solving the input-output equation to uniquely determine the parameter values is investigated.

The main component of algebraic identifiability is the existence of the algebraic equation.
However, in order for the equation to have a solution for the parameter, different assumptions are made. 
In the original definition \cite{diop1991nonlinear}, which was for continuous-time systems, the parameter was supposed to be algebraic over the differential field of the input and output, which implies that $\bm{\Phi}$ is a polynomial of the input, output and their derivatives and the parameter. 
In some later work \cite{nomm2016further ,xia2003identifiability}, both for continuous and discrete-time systems, a different assumption was made to enforce the existence of a solution: the Jacobian of the function $\bm{\Phi}$ with respect to the parameters is supposed non-singular. 
This second definition is particularly useful when the input-output equation is not algebraic (polynomial); 
then by checking the Jacobian, which is often straightforward, identifiability is verified. 
However, neither a differential field nor the Jacobian are defined in the discrete space. 
One can use the equivalence \emph{difference field} instead of the differential field, resulting in an almost identical definition of algebraic identifiability in the discrete space.
However, for the Jacobian matrix, the extension is not as straightforward.

\subsection{Algebraic identifiability based on the difference field}
We first present a \emph{difference field} \cite{ritt1939ideal} that is the equivalence discrete-time version of a differential field \cite{diop1991nonlinear, anstett2020priori}. 
\begin{definition}[\cite{ritt1939ideal}] 
    \label{def_difference field}
    A difference field $K$ is a pair consisting of a commutative field $k$ and an automorphism $\delta:k \to k$ such that for all $a,b \in k$, it holds that
    $\delta(a+b)=\delta(a)+\delta(b)$ and $\delta(ab)=\delta(a)\delta(b)$.
\end{definition}

Let $q$ be the \emph{shift operator}, defined by $q^T \h(t) = \h(t-T)$ for any function $\h:\R\to\R^\mtn, \mtn\in\Z_{>0}$, where $T\in\Z_{>0}$.

\begin{definition} \label{def_differenceFieldWithShiftOperator}
    A shift-operator (difference-)field $K$ is a difference field with the shift operator as the automorphism. 
\end{definition}

The notation $K\langle \h(t) \rangle$ (or $K\langle \h \rangle$) is a polynomial ring, consisting of all polynomials of $\h(t)$ with coefficients from the field $K$. 
If $K$ is a shift-operator difference field, then the coefficients can turn $\h(t)$ to $\h(t-1), \h(t-2), \ldots$.
As a result, $K\langle \h(t) \rangle$ would include polynomials of $\h(t),\h(t-1),\ldots$.
\begin{definition}[in the sense of \cite{xia2003identifiability}]
\label{definition_Alg_IDWithDifferenceField}
    Let $K$ be a shift-operator difference-field.
    System \eqref{equation_model} is \textbf{algebraically identifiable} 
    if $\t$ is (transformally) algebraic over $K\langle \u,\y\rangle$.
\end{definition} 

\cref{definition_Alg_IDWithDifferenceField} means that $\t$ satisfies a non-zero polynomial $\bm \Phi$ of itself and $\u(t)$ and $\y(t)$ and their iterates:
    $
    \bm \Phi(\t, \langle \u(t) \rangle_{0}^{T-1},\langle \y(t,\langle \u(t') \rangle_{0}^{t-1},\t) \rangle_0^T)=\bm{0}
    $
for some $T>0$.
Similar to the continuous case, $\bm\Phi$ being a polynomial guarantees the existence of a finite number of solutions for $\t$, and hence, it is locally identifiable. 

\subsection{Algebraic identifiability using discrete Jacobian}
We define a \textit{discrete Jacobian matrix} using finite differences. 
The $i^{\text{th}}$ entry of a vector $\bm X\in\R^{\mtn}$ is denoted by $X_i$. 
\begin{definition}[Forward finite difference]
    Given the scalar function $\Phi:\D\to\R,\D\subseteq \R^{\mtn},\mtn\in\Z_{\geq1},$ the \emph{forward finite difference along $\mtj$ by $r$ units at $\t\in\D$}, where 
    $\mtj\in\{\mathtt{1},\ldots,\mtn\}$ and $r$ is such that $\t+r\e_{\mtj}\in\D$, is defined by
    $
        \Delta^{r,\mtj} \Phi (\t) 
        \triangleq \Phi(\t + r\e_{\mtj})-\Phi(\t) 
    $
    where $\e_{\mtj}$ is the $\mtj^{\text{th}}$ column of the $n\times n$ identity matrix.
\end{definition}

\begin{definition}[Discrete Jacobian]
    Given function $\bm\Phi:\D\to\R^{\mtm},\D\subseteq \R^{\mtt},\mtm,\mtt\in\Z_{\geq1},$
    point $\t\in \D$, 
    and vector $\r \in \R^{\mtt}$ satisfying $\t+\r\in\D$,  
    the \emph{discrete Jacobian of $\bm\Phi$ by $\r$ with respect to $\t$} is defined by 
    \begin{equation*}
        \Delta^{\r} \bm\Phi(\t) 
        \triangleq 
        \begin{bmatrix}
            \Delta^{r_{\mathtt{1}},\mathtt{1}} \Phi_{\mathtt{1}}(\t)  & \dots & \Delta^{r_{\mtt},\mtt} \Phi_1(\t) \\
            \vdots &\ddots&\vdots \\ 
            \Delta^{r_{\mathtt{1}},\mathtt{1}} \Phi_{\mtm}(\t)  & \dots & \Delta^{r_{\mtt},\mtt} \Phi_{\mtm}(\t) \\
        \end{bmatrix}.
   \end{equation*}
\end{definition}

Equivalently, the Jacobian  can be written as 
    $
        \Delta^{\r} \bm\Phi(\t)
        =[\bm\Phi(\t+r_{\mathtt{1}}\e_{\mathtt{1}}), \ldots, \bm\Phi(\t+r_{\mtt}\e_{\mtt})]
        - [\bm\Phi(\t), \ldots, \bm\Phi(\t)].
    $
Similar to continuous-space Jacobian, the discrete Jacobian can be applied to a multivariate function.
Then the notation $\Delta^{\r}_{\t}$ is used to indicate that the Jacobian is applied with respect to the variable $\t$.
The need of having a non-singular Jacobian in the algebraic identifiability definitions in the continuous case is to use the Implicit Function Theorem.
The idea is that for arbitrary points $\t'$ in a small enough neighborhood of some reference point $\t=\bm{0}$, where the Jacobian is non-singular, Taylor's expansion can be used to approximate $\bm\Phi(\t')$ as $\bm\Phi(\bm{0})$ plus the Jacobian times $\t'$, and by taking the inverse of the Jacobian, $\t'$ can be written in terms of $\bm\Phi(\t')$, implying one-to-oneness in that neighborhood. 

In the discrete space, however, the higher order terms in Taylor's expansion can not be generally ignored, because the neighborhood may not be small enough.
Nevertheless, if $\bm\Phi$ is separable, then we can write $\bm\Phi(\t')$ in exact terms of $\bm\Phi(\t)$ for an arbitrary $\t$, which can result in the injectivity of $\bm\Phi$ as stated in \cref{lemma_oneToOneMapping}.
In what follows, let $\bm{0}$ (resp. $\bm{1}$) be the all-zero (resp. one) vector with the appropriate dimension.
\begin{definition}[\cite{viazminsky2008necessary}]
\label{defintion_SeparableFunctions}
    Function $\bm\Phi:\D\to\R^{\mtm},\D\subseteq \R^{\mtt},\mtm,\mtt\in\Z_{\geq1},$ is \emph{(additively) separable} if there exist $\mtt$ functions 
    $\g_i:\R\to\R^{\mtm}$, $i=\mathtt{1},\ldots,\mtt$, 
    such that    $\bm\Phi(\t)=\sum_{i=\mathtt{1}}^{\mtt}\g_i(\theta_i)$ for all $\t\in\D$.
\end{definition}
\begin{lemma}\label{lemma_oneToOneMapping}
    Consider an additively separable function $\bm\Phi:\D\to\R^{\mtm},\D\subseteq \R^{\mtt},\mtm,\mtt\in\Z_{\geq1}$. 
    For all $\t \in \D$ and
    $\r\in\R^{\mtt}\setminus\{\bm{0}\}$ satisfying $\t+\r\in \D$,
    $\Delta^{\r}\bm\Phi(\t) \bm{1} \neq \bm{0}$
    if and only if  $\bm\Phi$ is \textbf{one-to-one} on $\D$.
\end{lemma}

See the Apendix the proof.
Using the discrete Jacobian, we provide the discrete version of algebraic identifiability defined in \cite[Definition 5]{nomm2016further}.
\begin{definition}[in the sense of \cite{nomm2016further}]
    \label{def_AlgId}
    System \eqref{equation_model} is algebraically identifiable
    if there exist a positive integer $T$, subsets 
    $\U' \subset \U^{T}$,
    $\mathbf{\Theta}' \subset \mathbf{\Theta}$, 
    and an additively separable function $\bm \Phi: \bm{\Theta} \times \U'\times \Y^{T+1}  \rightarrow \R^{\mtt}$ with respect to its first argument,
     such that
            $ \bm \Phi(\t, \u^*,\langle \y(t,\u^*,\t) \rangle_0^T)=\bm{0}$ and  
            $\Delta^{\r}_{\t} \bm \Phi \bm{1} \neq \bm{0}$ 
        for all $(\t, \u^*) \in \bm\Theta' \times \U'$ 
        where $\u^*=\langle\u(t)\rangle_0^{T-1}$, 
        and for all $\r\in\R^{\mtt}\setminus\{\bm{0}\}$ satisfying $\t+\r \in \mathbf{\Theta}'$.
\end{definition}

The separability condition on $\bm\Phi$ means that  $\bm\Phi(\t,\cdot,\cdot)$ is additively separable with respect to $\t$.
The Jacobian approach in \cref{def_AlgId} may not be always of a great use as sometimes it requires the same computational effort compared to when injectivity is verified exhaustively for every point in the domain of $\bm\Phi$. 
Moreover, it can be applied only to additively separable functions.
These caveats motivate us to provide the following third definition of algebraic identifiability.

\subsection{Algebraic identifiability based on injectivity}

\begin{definition}
\label{definition_Alg_ID with 1-1}
    System \eqref{equation_model} is \textbf{algebraically identifiable} if there exist a positive integer $T$, subsets 
    $\U' \subset \U^T$,
    $\mathbf{\Theta}' \subset \mathbf{\Theta}$, 
    and a function 
    $\bm \Phi: \bm{\Theta} \times \U' \times \Y^{T+1}  \rightarrow \R^{\mtt}$ such that for all $(\t, \u^*) \in \mathbf{\Theta}'\times\U'$,
    $
             \bm \Phi(\t, \u^*,\langle \y(t,\u^*,\t) \rangle_0^T)=\bm{0} 
    $
     and for all $(\u^*,\y^*)\in\U'\times\Y^{T+1}$,
    $\bm \Phi(\cdot, \u^*,\y^*)$ \ is one-to-one in  $\bm\Theta'$.
\end{definition} 

The one-to-one condition guarantees the existence of an inverse function that allows to obtain $\t$ uniquely in terms of inputs and outputs and their time iterations.

    When the initial conditions are known, \cref{def_AlgId,definition_Alg_ID with 1-1} change to the so-called \textbf{algebraic identifiability with known initial conditions} \cite{nomm2016further}.
    The following extends \cref{definition_Alg_ID with 1-1}.
\begin{definition}\label{definition_Alg_ID with initials and 1-1}
    System \eqref{equation_model} is \textbf{algebraically identifiable with known initial conditions} if there exists a positive integer $T$, subsets 
    $\X' \subset \X$, 
    $\U' \subset \U^{T}$, and
    $\mathbf{\Theta}' \subset \mathbf{\Theta}$, 
    and a function 
    $\bm \Phi: \bm{\Theta} \times \X' \times \U' \times \Y^{T+1}  \rightarrow \R^{\mtt}$
    such that for all $(\t,\x_0,\u^*)\in\mathbf{\Theta}'\times\X'\times\U'$,
        \begin{equation}
            \label{equation_Alg-Id-with initial state}
             \bm \Phi(\t, \x_0, \u^*,\langle \y(t,\x_0,\u^*,\t) \rangle_0^T)=\bm{0} 
        \end{equation}
        and for all $(\x_0,\u^*,\y^*)\in\X'\times\U'\times\Y^{T+1},$
            $\bm \Phi(\cdot, \x_0, \u^*,\y^*)$ is one-to-one in $\bm\Theta'$.
\end{definition}

\setcounter{example}{0}
\begin{example}[revisited]
\label{ex_AlgID_Using_Jacobian_For_Coordinating_Example}
    Let $n=4$ and consider an initial condition such that $y(0) = \tfrac{1}{4}$.
    The input-output equation is
    \begin{equation}\label{eq_InputOutputRelationOfTheExample}
        y(t)=\theta_1 h(y(t-1);\tfrac{1}{4})+ \theta_2 h(y(t-1);\tfrac{1}{2}),
    \end{equation}
    When $\theta_1\geq\frac{1}{2}$, we have $y(1)=\theta_1$ and $y(2)=\theta_1+\theta_2$.
    Hence, the function 
    $\bm \Phi: \bm{\Theta} \times \Y^{3}  \rightarrow \R^{\mtt}$ with 
    $
        \bm \Phi(\t,\langle y(t) \rangle_0^2)
        =[
            \theta_1-y(1),
            \theta_1+\theta_2 - y(2)
       ]^\top
    $ 
    is additively separable w.r.t. $\t$, and
     \vspace{-7pt}
    \[
        \Delta^{\r}_{\t}\bm\Phi (\t,\langle y(t) \rangle_0^2)\bm{1}=
        \begin{bmatrix}
        r_1 & 0 \\
        r_1 & r_2\\
        \end{bmatrix}
        \begin{bmatrix}
            1 \\ 1
        \end{bmatrix}
        \neq \bm {0}
    \]
    for all $\bm {r} \neq \bm 0$.
    Hence, the system is algebraically identifiable with known initial conditions.
    When $\theta_1< \frac{1}{2}$, the result still holds as the function reduced to $
        \bm \Phi(\t,\langle y(t) \rangle_0^2)
        =[
            \theta_1-y(1),
            \theta_1 - y(2)
       ]^\top
    $ with the discrete Jacobian
    $ \begin{bmatrix}
        r_1 & 0 \\
        r_1 & 0\\
        \end{bmatrix}$.
\end{example}
\begin{remark} \label{re_equivalenceOfDefinitions}     
    In view of \cref{lemma_oneToOneMapping}, \cref{def_AlgId} implies \cref{definition_Alg_ID with 1-1}, but not vice versa, because $\bm\Phi$ in \cref{definition_Alg_ID with 1-1} does not need to be separable.
\end{remark} 
The main challenge in algebraic approaches is deriving an input-output equation.
Besides ad-hoc substitutions and iterations 
to exclude the state variable, as described in \cite{eisenberg2013input}, Ritt's algorithm and Groebner bases approach can be used when system \eqref{equation_model} consists of rational functions.
\cite{ljung1994global,anstett2006chaotic,calandrini1997ritt}.\\

\textbf{Step 1}: \textit{Iterate the system equations \eqref{equation_model} to obtain a number of equations strictly greater than the number of state variables and their time iterates used in \eqref{equation_model}.}\\
\textbf{Step 2}: \textit{Calculate the Groebner basis with the \emph{lexicographic ordering}
$\u(t) \prec  \ldots \prec \u(t+\alpha) \prec \y(t) \prec \ldots \prec \y(t+\beta) \prec \x(t) \prec\ldots \prec\x(t+\gamma)$ to eliminate state variable $\x$ and obtain the input-output relation.
This can be done using software such as Maxima and Maple \cite{anstett2006chaotic}.}

\begin{remark}\label{re_localObservability}
Examining the \textit{Local observability condition} is popular in identifiability analysis tools developed for both continuous and discrete-time systems \cite{anstett2020priori,sciml}.
This criterion rests on the assumption that the output function is smooth w.r.t the state variables, which is not required in  the proposed approaches in \cref{def_AlgId,definition_Alg_ID with 1-1,definition_Alg_ID with initials and 1-1}, and \cref{Theorem-input-output-relation}.
\end{remark}
\vspace{-1pt}
\section{Results}\label{sec_results}
\subsection{Local result}
The following theorem links the notion of algebraic identifiability to that of local structural identifiability.

\begin{theorem}\label{theorem_relation between ALg Id and structural Id} 
    System \eqref{equation_model} is locally structurally identifiable if and only if it is algebraically identifiable with known initial conditions (defined by \cref{definition_Alg_ID with initials and 1-1}).
\end{theorem}
\begin{proof} 
    \textit{(sufficiency)}
    Since the system is structurally identifiable,
    there exists 
    a time $T>0$ 
    and subsets 
    $\mathbf{\Theta}'\subseteq\mathbf{\Theta}$, 
    $\X'\subseteq\X$, and
    $\U'\subseteq\U^T$,
    such that 
    the outputs of any distinct $\t, \t' \in \mathbf{\Theta}'$ are different at time $T$, i.e., 
    $ \y(\f^T(\x_0,\u(0),\dots,\u(T-1)),\t)
    \neq\y(\f^T(\x_0,\u(0),\dots,\u(T-1)),\t')$,
    for every initial condition $\x_0 \in \X'$ and input sequence
    $\langle\u(t)\rangle_0^{T-1}\in\U'$.
    Thus, function $\bm\Psi:\mathbf{\Theta}'\times\X'\times\U' \to \Y^{T+1}$ defined by 
    $
    \bm\Psi(\t,\x_0,\langle\u(t)\rangle_0^{T-1})
     =
    [
         \y(\x_0,\t),
         \y(\f(\x_0,\u(0)),\t),
         \ldots,
         \y(\f^T(\x_0,\u(0),\dots,\u(T-1)),\t)
    ]^\top
    $
    is one-to-one w.r.t $\t$ in $\bm\Theta'$.
    Thus, for every $(\x_0,\u^*) \in \X'\times\U'$,
    where $\u^*=\langle\u(t)\rangle_0^{T-1}$, there exists the one-to-one function $\bm g_{\x_0,\u^*}: \bm\Theta'\to\Y^{T+1}$ with 
    $\bm g(\bm \theta)=\bm\Psi(\bm\theta,\x_0,\u^*)$ 
    Hence, the inverse function $\bm g_{\x_0,\u^*}^{-1}: \Y^{T+1}\to \bm\Theta'$
    exists, where 
    $\bm{g}^{-1}_{\x_0,\u^*}(\y^*) = \t$
    for $\y^* = \langle\y(t,\u^*,\t)\rangle_0^T$.
    Now consider the function 
    $\bm \Phi: \bm{\Theta} \times \X\times \U' \times \Y^{T+1}  \rightarrow \R^{\mtt}$ defined by
    $
        \bm\Phi(\t,\x_0,\u^*,\y^*) 
            = \t 
            - \bm{g}^{-1}_{\x_0,\u^*}(\y^*).
    $
    Clearly, \eqref{equation_Alg-Id-with initial state} is satisfied when $\y^*$ satisfies the dynamics of System \eqref{equation_model}.
    On the other hand, $\bm\Phi$ is one-to-one w.r.t. $\t$, because for distinct $\t,\t'\in\bm\Theta'$, the equality $\t - \bm{g}^{-1}_{\x_0,\u^*}(\y^*) = \t' - \bm{g}^{-1}_{\x_0,\u^*}(\y^*)$ implies $\t=\t'$ for all $(\x_0,\u^*,\y^*)\in\X\times\U'\times\Y^{T+1}$.
    \textit{(necessity)}
    Since the system is algebraically identifiable with known initial conditions, there exists a positive integer $T$, subsets
    $\X' \subset \X$,
    $\U' \subset \U^{T}$, 
    $\mathbf{\Theta}' \subset \mathbf{\Theta}$, 
    and a function $\bm\Phi$, satisfying \eqref{equation_Alg-Id-with initial state} and one-to-one in $\mathbf{\Theta}'$.
    Consider parameters $\t,\t'\in\bm\Theta'$. 
    Should they result in the same output sequence, i.e., $\langle\y(t,\x_0,\u^*,\t)\rangle_0^T = \langle\y(t,\x_0,\u^*,\t')\rangle_0^T = \y^*$ for some input sequence $\u^*\in\U'$ and initial condition $\x_0\in\X'$, it follows that 
    $\bm\Phi(\t,\x_0,\u^*,\y^*) = \bm\Phi(\t',\x_0,\u^*,\y^*) = 0$.
    Since $\bm\Phi$ is one-to-one w.r.t. $\t$ in $\bm\Theta'$, we have that $\t = \t'$.
    Thus, the system is locally structurally identifiable.
\end{proof}

\setcounter{example}{0}
\begin{example}[revisited]\label{ex_ForLocalResults}
    We showed that system \eqref{eq:Coordinators_Dynamics} is locally structurally identifiable and algebraically identifiable with known initial conditions, illustrating \cref{theorem_relation between ALg Id and structural Id}.
\end{example}

\subsection{Global results}
If system \eqref{equation_model} satisfies a set of linear regressions w.r.t. $\t$ where the regressors are in terms of inputs, outputs, and their iterations, global identifiability is guaranteed.
The following is an extension of the results in \cite{anstett2006chaotic} for discrete-time systems, which itself was an extension of the work in \cite{ljung1994global}.
The proof is straightforward, which is perhaps why it was stated as a definition in \cite{anstett2020priori}.
\begin{proposition}\label{theorem_globalIdentifiability}
    System \eqref{equation_model} is \textbf{globally identifiable} if 
   it satisfies a set of $\mtt$ linear equations
   \vspace{-6pt}
   \begin{equation}\label{equation_LinearRegression}
        \scalebox{0.91}{
        $ P_i(\langle \u(t)\rangle_0^{T-1}, \langle \y(t)\rangle_0^T)\theta_i- Q_i(\langle \u(t)\rangle_0^{T-1}, \langle \y(t)\rangle_0^T)=0 $ 
        }
    \end{equation}
    for $i = 1,\dots,\mtt$, where $Q_i:\U^T \times \Y^{T+1} \to \R$, $P_i:\U^T \times \Y^{T+1} \to \R$, and $P_i$s are not identically zero.
\end{proposition}

Obtaining the set of equations \eqref{equation_LinearRegression} is not always guaranteed. 
Assume that, instead, by excluding the state variables in \eqref{equation_model}, we obtain the following input-output equation:
\vspace{-3pt}
\begin{equation}
    \label{equation_Alpha-Input-output-Eq}   
    \textstyle
    \sum_{i=0}^r\alpha_i(\t)\phi_i(\langle \u(t)\rangle_0^{T-1}, \langle \y(t)\rangle_0^T)=0
    \end{equation}
where $\phi_i:\U^T \times \Y^{T+1} \to \R$, $i=0,\ldots, r$, $r\in\Z_{\geq1}$, is a function of inputs, outputs, and their iterations for some large enough $T$, and  $\alpha_i:\mathbf{\Theta} \to  \R$, $i = 0,\ldots,r$, is the coefficient map of $\phi_i$  with $\alpha_0 = 1$.
By iterating \eqref{equation_Alpha-Input-output-Eq} over time for $r-1$ times, we obtain the system of linear equations
\vspace{-2pt}
\begin{equation}
\label{equation_Rearranged_Alpha_Input-Output-Eq}
    \scalebox{.85}{$
\underbrace{\begin{bmatrix}
    \phi_1^0 & \dots &  \phi_r^0\\
    \vdots &\ddots&\vdots \\ 
    \phi_1^{r-1} & \dots &  \phi_r^{r-1}\\
    \end{bmatrix}}_{\A} 
    \underbrace{
    \begin{bmatrix}
        \alpha_1\\ \vdots \\ \alpha_r
    \end{bmatrix}}_{\a(\t)}
    = -          
    \underbrace{
    \begin{bmatrix}
    \phi_0^0\\
    \vdots\\
    \phi_0^{r-1}
    \end{bmatrix}}_{\b},
    $}
\end{equation}
where  $\phi_i^s=\phi_i(\langle \u(t)\rangle_s^{T-1+s}, \langle \y(t)\rangle_s^{T+s})$ for $s=0,\ldots,r-1$.
If $\det \A \neq 0$,
the coefficients can then be determined uniquely as $\a(\t) = -\A^{-1}\b$, where $\a$ is known as the \textit{exhaustive summary}.
Thus, the identifiability problem is reduced to checking the injectivity of the exhaustive summary \cite{saccomani2003parameter}.
In practice, due to noise in data, the parameters may not be precisely estimated using \eqref{equation_Rearranged_Alpha_Input-Output-Eq}, but these estimates can provide good initial values for parameter optimization algorithms \cite{ljung1994global, verdiere2005identifiability}.
\begin{theorem} \label{Theorem-input-output-relation}
  System \eqref{equation_model} is \textbf{globally identifiable} if there is an input-output relation in the form of \eqref{equation_Alpha-Input-output-Eq} such that
  \begin{enumerate}
      \item the derived matrix $\A$ in \eqref{equation_Rearranged_Alpha_Input-Output-Eq} has full rank,
      and
      \item $\bm \a:\mathbf{\Theta} \rightarrow \R^r$ is one-to-one. 
  \end{enumerate}
\end{theorem}

\setcounter{example}{3}
\begin{example}\label{example_Theorem_2}
Consider the system given by 
$\x(t+1)=u^2(t)$ and $y(t)=\theta x^2(t)$
where 
$\x \in \Z$, 
$u \in \Z_{>0}$, $y\in \Z$, and 
$\t \in \Z$.
Clearly, $\A = u^2(0)$ is full rank and $\alpha_1=\theta$ is one-to-one.
Thus, the system is globally identifiable.
\end{example}

The reverse of \cref{Theorem-input-output-relation} does not always hold. Counter examples exist for multi-output systems, non-minimal parameter spaces, and systems for which matrix A cannot be constructed.

\setcounter{example}{4}
\begin{example}
\label{example_counter_example_for_Thm_2}
Consider the system 
\vspace{-5pt}
$$
\begin{cases}
    & \x(t+1)=[u(t) , 2u(t)]^\top \\
    & y(t)=\theta_1x_1(t)+\theta_2x_2(t)
\end{cases}   
$$
where 
$\x \in \Z^2$, 
$u \in Z$, $y\in \Z$, and 
$\t \in \mathds{R}^2_{\geq 0}$
with $\theta_1+\theta_2=1$.
By excluding the state variable and iterating over time we get
\vspace{-5pt}
\begin{equation*}
        \underbrace{
        \begin{bmatrix}
            u(1) & 2u(1)\\
            u(2) & 2u(2)
        \end{bmatrix}}_{\A}
        \underbrace{
        \begin{bmatrix}
            \alpha_1 \\ \alpha_2
        \end{bmatrix}}_{\a(\t)}
        =          
        \underbrace{
        \begin{bmatrix}
            y(2)\\
            y(3)
        \end{bmatrix}}_{\b},
    \end{equation*}
where $\alpha_1=\theta_1$ and $\alpha_2=\theta_2$.
Matrix $\A$ is not full rank; however, the system is globally identifiable as $\theta_1$ and $\theta_2$ can be uniquely obtained by substituting $\theta_1=1-\theta_2$.
\end{example}

\begin{example}
    Consider the following system with two states and two outputs and four parameters
    $$
    \begin{cases}
        &x_1(t+1) = u(t) +k_{12}x_2(t) - (k_{01} + k_{21})x_1(t) \\
       & x_2(t+1) = k_{21}x_1(t) - (k_{02} + k_{12})x_2(t) \\
       & y_1(t)  = x_1(t) \\
        &y_2(t)  = x_2(t).
    \end{cases}  
    $$
The system admits the following input-output equation
    \begin{equation}
        u(t) - k_{01} y_1(t) -y_1(t+1)  - k_{02} y_2(t) -y_2(t+1) = 0.
    \end{equation}
By iterating the input-output equation three times--to get as many equations as the number of parameters, the matrix $\A$ can be constructed.
Yet, matrix $\A$ is not full rank. 
However, it has been shown that the number of input-output equations is equal to the number of outputs\cite{eisenberg2013input}, and for this system are as follows
\begin{equation}
    \begin{aligned}
       - y_1(t+1) &+ u(t) +k_{12}y_2(t) - (k_{01} + k_{21})y_1(t) =0, \\
       - y_2(t+1) &+ k_{21}y_1(t) - (k_{02} + k_{12})y_2(t) =0.
    \end{aligned}
\end{equation}
 The coefficient of each of these two equations is globally identifiable \cite{eisenberg2013input}[Theorem 4.1]. 
The exhaustive summary would be then
$\alpha_1 = k_{12}$, $\alpha_2 = k_{01} + k_{21}$, $\alpha_3 = k_{21}$, and $\alpha_4 = k_{02} + k_{12}$. The injectivity of the exhaustive summary is immediate which implies the global identifiability of the parameters.    
\end{example}

\begin{example} Consider the system given by
    $$
    \begin{cases}
       &x(t+1) = x(t) -sgn(x(t) - \theta)u(t) \\\
        &y(t)=x(t) 
    \end{cases}
    $$
for $u,x,\theta \in \Z$.
This system can not be written in the format of \eqref{equation_Alpha-Input-output-Eq} and also in \eqref{equation_Rearranged_Alpha_Input-Output-Eq}. 
Thus, Theorem 2 can not be applied. 
But we can show that this system is identifiable analytically.
\end{example}

\begin{remark}\label{re_On_ID_Parameter_Combinations}
   Verifying the injectivity of the exhaustive summary is not straightforward. 
   If it is a rational function of the parameters, Buchberger's algorithm can be applied \cite{saccomani2003parameter}.
\end{remark}

\vspace{-4pt}
\section{Concluding remarks}
We extended identifiability notions from continuous to discrete-space systems. While analytical definitions easily transferred, algebraic definitions faced challenges due to the loss of differentiability. We presented three definitions for algebraic identifiability in discrete space, including one using the discrete Jacobian matrix, and discussed their relation to analytical definitions. We also presented global identifiability results, however, their utility may not always match that of continuous space.
 \vspace{-8pt}
\bibliographystyle{IEEEtran}
\bibliography{IEEEabrv,ref}

\section*{Appendix} 

\begin{proof}[Proof of Lemma 1]
\label{Proof_lemma 1}  
    Let $\t,\t' \in \bm\Theta$ and $\r=\t'-\t$.
    Since $\bm\Phi$ is separable, functions $\g_i$ exist as in \cref{defintion_SeparableFunctions}, such that
    \begin{align*}
        \bm\Phi(\t')
        &= \sum_{i=\mathtt{1}}^{\mtt} \g_i(\theta_i+r_i) \\
        & = \sum_{i=\mathtt{1}}^{\mtt} \g_i(\theta_i)+\sum_{i=\mathtt{1}}^{\mtt} \g_i(\theta_i+r_i) - \g_i(\theta_i) \\
        &= \bm\Phi(\t)\scalebox{0.85}{$
         +\sum_{i=\mathtt{1}}^{\mtt} \Big(\g_i(\theta_i+r_i) - \g_i(\theta_i)
          +\sum_{j\neq i} \g_j(\theta_j) - \g_j(\theta_j)\Big)$}\\
        &= \bm\Phi(\t) +\sum_{i=\mathtt{1}}^{\mtt} \big(\bm\Phi(\t+r_i\e_i) - \bm\Phi(\t)\big) \\
        &= \bm\Phi(\t) +\Delta^{\r} \bm\Phi(\t) \bm{1}.
    \end{align*}    
    (sufficiency)
    Assume on the contrary that there exists some $\r\neq \bm{0}$ such that $\Delta^{\r}\bm\Phi(\t) \bm{1} = \bm{0}$. 
    Then $\t' \neq \t$ but 
    $\bm\Phi(\t')=\bm\Phi(\t)$, which violates the one-to-one assumption.
    Thus, $\Delta^{\r}\bm\Phi(\t) \bm{1} \neq \bm{0}$ for all non-zero $\r$.
    (necessity)
    Should $\bm\Phi(\t')=\bm\Phi(\t)$,
    then $\Delta^{\r} \bm\Phi(\t) \bm{1} = \bm 0$, which implies that $\r=\bm{0}$ and in turn $\t'=\t$.
    Thus, $\bm\Phi$ is one-to-one.
\end{proof}

\end{document}